\documentclass{paper}
\usepackage[utf8]{inputenc}
\usepackage{float}
\usepackage{amsmath}
\usepackage{amsthm}
\usepackage{amssymb}
\usepackage{graphicx}
\usepackage[numbers]{natbib}
\usepackage[unicode=true,
 bookmarks=false,
 breaklinks=false,pdfborder={0 0 1},backref=section,colorlinks=false]
 {hyperref}

\makeatletter

\newcommand{\noun}[1]{\textsc{#1}}

\theoremstyle{plain}
    \ifx\thechapter\undefined
      \newtheorem{prop}{\protect\propositionname}
    \else
      \newtheorem{prop}{\protect\propositionname}[chapter]
    \fi

\usepackage{stmaryrd}

\makeatother

\providecommand{\propositionname}{Proposition}

\begin{document}
\title{Karush–Kuhn–Tucker conditions to build efficient contractors\\
Application to TDoA localization}
\author{Luc Jaulin}
\institution{Lab-Sticc, ENSTA-Bretagne}

\maketitle
\textbf{Abstract}. This paper proposes an efficient contractor for
the TDoA (Time Differential of Arrival) equation. The contractor is
based on a minimal inclusion test which is built using the Karush–Kuhn–Tucker
(KKT) conditions. An application related to the localization of sound
sources using a TDoA technique is proposed. 

\section{Introduction}

To solve nonlinear problems with nonlinear constraints, a classical
approach is based on the Karush–Kuhn–Tucker (KKT) conditions \citep{Karush:39}
\citep{Kuhn:Tucker:51}. To use the KKT conditions, we first have
to formulate our problem as an optimization problem in standard form: 

\begin{equation}
\begin{array}{c}
\text{minimize}f(\mathbf{x})\\
\text{s.t.}\,\,\mathbf{g}(\mathbf{x})\leq0
\end{array}
\end{equation}
where function $f:\mathbb{R}^{n}\mapsto\mathbb{R}$ and $\mathbf{g}:\mathbb{R}^{n}\mapsto\mathbb{R}^{m}$
are assumed to be differentiable. We then build the Lagrangian 
\begin{equation}
\begin{array}{ccc}
\mathcal{L}(\mathbf{x},\boldsymbol{\mu}) & = & f(\mathbf{x})+\boldsymbol{\mu}^{\text{T}}\cdot\mathbf{g}(\mathbf{x}).\end{array}
\end{equation}
The necessary conditions for $\mathbf{x}$ to be an optimizer are:
\[
\left\{ \begin{array}{ccl}
\frac{df}{d\mathbf{x}}(\mathbf{x})+\sum_{i}\mu_{i}\frac{dg_{i}}{d\mathbf{x}}(\mathbf{x})=0 &  & \text{(stationarity)}\\
g_{i}(\mathbf{x})\leq0,\forall i &  & \text{(primal feasibility)}\\
\mu_{i}\geq0 &  & \text{\text{(dual feasibility)}}\\
\mu_{i}g_{i}(\mathbf{x})=0,\forall i &  & \text{(complementary slackness)}
\end{array}\right.
\]

The problem can be interpreted as moving a particle at position $\mathbf{x}$
in the space, with two kinds of forces:
\begin{itemize}
\item $f$ is a potential and the force generated by $f$ is $-\text{grad}f$. 
\item The constraint $g_{i}(\mathbf{x})\leq0$ corresponds to reaction forces
generated by one-sided constraint surfaces delimiting the free space
for $\mathbf{x}$. The particle is allowed to move inside $g_{i}(\mathbf{x})\leq0$,
but as soon as it touches the surface $g_{i}(\mathbf{x})=0$, it is
pushed inwards the free space. 
\end{itemize}
Stationarity states that $-\text{grad}f$ is a linear combination
of the reaction forces. Dual feasibility states that reaction forces
point inwards the free space for $\mathbf{x}$. Slackness states that
if $g_{i}(\mathbf{x})<0$ , then the corresponding reaction force
must be zero, since the particle is not in contact with the surface.

Interval methods have used these KKT conditions to solve nonlinear
optimization problems \citep{Moore79}\citep{Hansen92}\citep{Ratschek88}
when inequality constraints are involved. 

In this paper, we propose to use the KKT conditions to build minimal
inclusion tests in order to derive efficient contractors. For this,
we consider a constraint of the form $y=f(\mathbf{x})$ where $f:\mathbb{R}^{n}\rightarrow\mathbb{R}$
and we assume that $\mathbf{x}\in[\mathbf{x}]$, where $[\mathbf{x}]$
is an axis aligned closed box of $\mathbb{R}^{n}$. The feasible values
for $y$ is an interval $[y]=[y^{-},y^{+}]$ which can be obtained
by solving the two minimization problems
\begin{equation}
\begin{array}{ccc}
y^{-} & = & \min f(\mathbf{x})\\
 &  & \mathbf{x}\in[\mathbf{x}]
\end{array}
\end{equation}
and 
\begin{equation}
\begin{array}{ccc}
y^{+} & =- & \min(-f(\mathbf{x}))\\
 &  & \mathbf{x}\in[\mathbf{x}]
\end{array}
\end{equation}

As a consequence, the KKT conditions could be used as least for the
forward contraction, \emph{i.e.}, to contract the feasible interval
$[y]$ for $y$. These conditions can be treated either symbolically
for simple constraints or automatically with pessimism as in \citep{Hansen92}.

Section \ref{sec:TDoA-constraint} defines the TDoA constraint \citep{Lee:multilateration}
which will illustrate the benefit brought by the use of the KKT conditions.
Section \ref{sec:Action} introduces the notion of action of a contractor
on a separator. This notion will allow us build complex separator
using the composition with other constraints. Section \ref{sec:Application}
illustrates how the notion of action of a TDoA contractor on the separator
(obtained after data treatment) can be used to localize sound sources.
Section \ref{sec:Conclusion} concludes the paper. 

\section{TDoA constraint\label{sec:TDoA-constraint}}

The TDoA constraint is defined by
\begin{equation}
\|\mathbf{x}-\mathbf{a}\|-\|\mathbf{x}-\mathbf{b}\|=y\label{eq:TDoA}
\end{equation}
where $\mathbf{x}\in\mathbb{R}^{2}$ and $y\in\mathbb{R}$ are the
variables. The parameters $\mathbf{a}\in\mathbb{R}^{2},\mathbf{b}\in\mathbb{R}^{2}$
are assumed to be known. Equivalently, we have
\begin{equation}
f(\mathbf{x})=y
\end{equation}
where 
\begin{equation}
f(\mathbf{x})=\sqrt{(x_{1}-a_{1})^{2}+(x_{2}-a_{2})^{2}}-\sqrt{(x_{1}-b_{1})^{2}+(x_{2}-b_{2})^{2}}.\label{eq:f:TDoA}
\end{equation}
In this section, we want to build an efficient contractor for (\ref{eq:TDoA})
, \emph{i.e.}, given a box $[\mathbf{x}]\ni\mathbf{x}$ and an interval
$[y]\ni y,$ we want to contract $[\mathbf{x}]$ and $[y]$ without
removing a single pair $(\mathbf{x},y)$ of the constraint (see \citep{ChabertJaulin09}
for a formal definition of a contractor). We will mainly focus on
the forward contraction, \emph{i.e.}, the contraction of $[y]$ which
can be interpreted as an interval evaluation of $f$. Interval analysis
has already been used to solve problems involving the TDoA constraint
in \citep{ReynetJaulinChabert09}, \citep{DrevelleThese} and \citep{jaulin2023hyperbola}.

\textbf{Notation}. In what follows, $[\mathbb{A}]$ represents the
smallest closed interval which contains the set $\mathbb{A}\subset\mathbb{R}$.
When $\mathbb{A}\subset\mathbb{R}^{n},$ $[\mathbb{A}]$ denotes the
smallest axis-aligned box which contains $\mathbb{A}.$

\subsection{Interval evaluation}

The interval evaluation of $f([\mathbf{x}])$ over a box $[\mathbf{x}]$
can be obtained using the following proposition.
\begin{prop}
Given a non-degenerated box $[\mathbf{x}]\in\mathbb{R}^{2}$, and
the function $f$ given by (\ref{eq:f:TDoA}), we have
\begin{equation}
f([\mathbf{x}])=[f(\mathbb{P}_{0}\cup\mathbb{P}_{1}\cup\mathbb{P}_{2})]
\end{equation}
with 
\begin{equation}
\begin{array}{ccc}
\mathbb{P}_{0} & = & \{(x_{1}^{-},x_{2}^{-}),(x_{1}^{-},x_{2}^{+}),(x_{1}^{+},x_{2}^{-}),(x_{1}^{+},x_{2}^{+})\}\\
\mathbb{P}_{1} & = & \left\{ (x_{1},x_{2})\in\partial[x_{1}]\times[x_{2}]\,|\,x_{2}=\varphi_{1}(x_{1},\mathbf{a},\mathbf{b})\right\} \\
\mathbb{P}_{2} & = & \left\{ (x_{1},x_{2})\in[x_{1}]\times\partial[x_{2}]\,|\,x_{1}=\varphi_{2}(x_{2},\mathbf{a},\mathbf{b})\right\} 
\end{array}
\end{equation}
with
\begin{equation}
\begin{array}{ccc}
\varphi_{1}(x_{1},\mathbf{a},\mathbf{b}) & = & \frac{a_{2}\cdot|x_{1}-b_{1}|-b_{2}\cdot|x_{1}-a_{1}|}{|x_{1}-b_{1}|-|x_{1}-a_{1}|}\\
\varphi_{2}(x_{2},\mathbf{a},\mathbf{b}) & = & \frac{a_{1}\cdot|x_{2}-b_{2}|-b_{1}\cdot|x_{2}-a_{2}|}{|x_{2}-b_{2}|-|x_{2}-a_{2}|}
\end{array}
\end{equation}
\end{prop}
This proposition tells us that the extrema for $f$ are reached either
on the corners of $[\mathbf{x}]$ or on the edges of $[\mathbf{x}]$
but never in the interior of $[\mathbf{x}]$. The border operator
$\partial$ is used to get the bounds of the interval. For instance
$\partial[x_{1}]\times[x_{2}]=\{x_{1}^{-},x_{1}^{+}\}\times[x_{2}]=(\{x_{1}^{-}\}\times[x_{2}])\cup(\{x_{1}^{+}\}\times[x_{2}])$.
An illustration is given by Figure \ref{fig:tdoa_contour.png}. 

\begin{figure}[H]
\begin{centering}
\includegraphics[width=8cm]{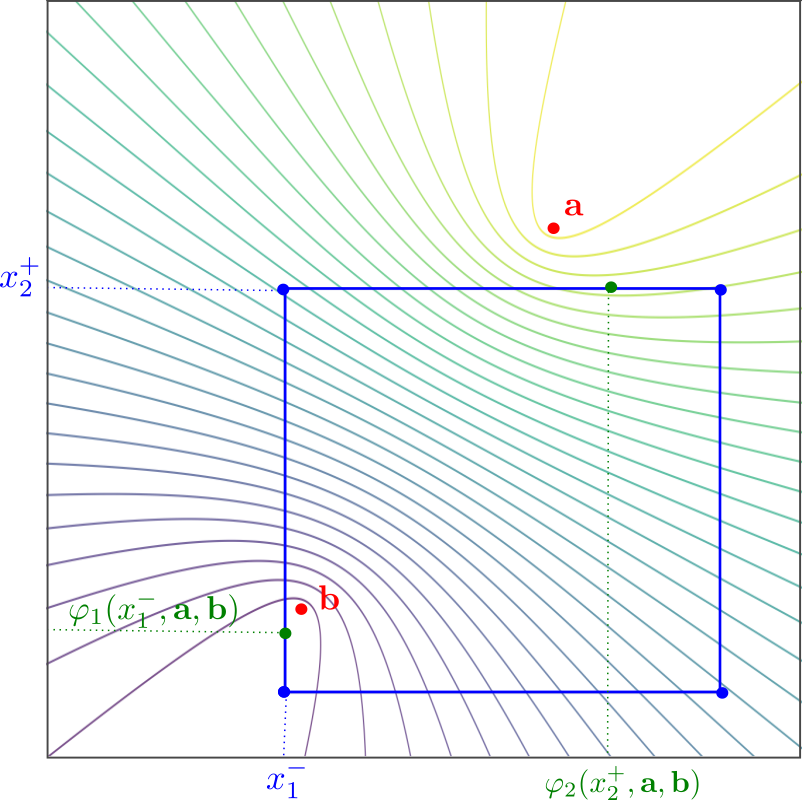}
\par\end{centering}
\caption{Level curve of the TDoA function $f$. The extrema of $f$ over a
box are reached either on the corners (blue) or on some specific points
of the edges (green)}
\label{fig:tdoa_contour.png}
\end{figure}

\begin{proof}
We need to solve
\begin{equation}
\begin{array}{c}
\text{Optimize}f(\mathbf{x})\\
\text{s.t.}\,\,\mathbf{g}(\mathbf{x})\leq0
\end{array}
\end{equation}
where 
\begin{equation}
\mathbf{g}(\mathbf{x})=\left(\begin{array}{c}
x_{1}-x_{1}^{+}\\
-x_{1}+x_{1}^{-}\\
x_{2}-x_{2}^{+}\\
-x_{2}+x_{2}^{-}
\end{array}\right).
\end{equation}
The Lagrangian is
\begin{equation}
\begin{array}{ccc}
\mathcal{L}(\mathbf{x},\boldsymbol{\mu}) & = & f(\mathbf{x})+\boldsymbol{\mu}\cdot\mathbf{g}(\mathbf{x})\\
 & = & \sqrt{(x_{1}-a_{1})^{2}+(x_{2}-a_{2})^{2}}-\sqrt{(x_{1}-b_{1})^{2}+(x_{2}-b_{2})^{2}}\\
 &  & +\mu_{1}(x_{1}-x_{1}^{+})-\mu_{2}(x_{1}-x_{1}^{-})+\mu_{3}(x_{2}-x_{2}^{+})-\mu_{4}(x_{2}-x_{2}^{-})
\end{array}
\end{equation}

The necessary conditions for $\mathbf{x}$ to be an optimizer are:
\begin{equation}
\left\{ \begin{array}{cc}
\begin{array}{c}
\frac{\partial f}{\partial x_{1}}(\mathbf{x})+\mu_{1}x_{1}-\mu_{2}x_{1}=0\\
\frac{\partial f}{\partial x_{2}}(\mathbf{x})+\mu_{3}x_{2}-\mu_{4}x_{2}=0
\end{array} & \text{(stationarity)}\\
g_{i}(\mathbf{x})\leq0,\forall i & \text{(primal feasibility)}\\
\mu_{i}\geq0 & \text{\text{(dual feasibility)}}\\
\mu_{i}g_{i}(\mathbf{x})=0,\forall i & \text{(complementary slackness)}
\end{array}\right.
\end{equation}
where
\begin{equation}
\begin{array}{ccc}
\frac{\partial f}{\partial x_{1}}(\mathbf{x}) & = & \frac{x_{1}-a_{1}}{\sqrt{(x_{1}-a_{1})^{2}+(x_{2}-a_{2})^{2}}}-\frac{x_{1}-b_{1}}{\sqrt{(x_{1}-b_{1})^{2}+(x_{2}-b_{2})^{2}}}\\
\frac{\partial f}{\partial x_{2}}(\mathbf{x}) & = & \frac{x_{2}-a_{2}}{\sqrt{(x_{1}-a_{1})^{2}+(x_{2}-a_{2})^{2}}}-\frac{x_{2}-b_{2}}{\sqrt{(x_{1}-b_{1})^{2}+(x_{2}-b_{2})^{2}}}
\end{array}
\end{equation}
\end{proof}
Since the box $[\mathbf{x}]$ is non degenerated (\emph{i.e.}, it
has a non empty interior), and from the complementary slackness condition,
an optimizer should correspond to one of the following situations:
\begin{itemize}
\item $\mathbf{x}$ is a corner, $(\mu_{1},\mu_{3})=(0,0)$ , $(\mu_{1},\mu_{4})=(0,0),$
$(\mu_{2},\mu_{3})=(0,0)$ or $(\mu_{2},\mu_{4})=(0,0),$
\item \textbf{$\mathbf{x}$} is an edge, $(\mu_{1},\mu_{2},\mu_{3})=(0,0,0)$,
$(\mu_{1},\mu_{2},\mu_{4})=(0,0,0)$, $(\mu_{1},\mu_{3},\mu_{4})=(0,0,0)$
or $(\mu_{2},\mu_{3},\mu_{4})=(0,0,0).$
\item \textbf{$\mathbf{x}$} is in the interior of $[\mathbf{x}]$, \emph{i.e}.,
$(\mu_{1},\mu_{2},\mu_{3},\mu_{4})=(0,0,0,0)$. 
\end{itemize}
\textbf{Case 1}. $\mathbf{x}$ is a corner. The optimizers are inside
the set 
\begin{equation}
\mathbb{P}_{0}=\{(x_{1}^{-},x_{2}^{-}),(x_{1}^{-},x_{2}^{+}),(x_{1}^{+},x_{2}^{-}),(x_{1}^{+},x_{2}^{+})\}.
\end{equation}
\textbf{Case 2}. \textbf{$\mathbf{x}$} is an edge. Take first $\mu_{2}=\mu_{3}=\mu_{4}=0$
which means that we consider the right edge of $[\mathbf{x}]$: $x_{1}-x_{1}^{+}=0$.
The other edges are deduced by symmetry. We get
\[
\begin{array}{cl}
 & \left\{ \begin{array}{c}
\frac{\partial f}{\partial x_{1}}(\mathbf{x})+\mu_{1}x_{1}=0\\
\frac{\partial f}{\partial x_{2}}(\mathbf{x})=0
\end{array}\right.\\
\,\\
\Rightarrow & \,\,\frac{\partial f}{\partial x_{2}}(\mathbf{x})=0\\
\,\\
\Leftrightarrow & \frac{x_{2}-a_{2}}{\sqrt{(x_{1}-a_{1})^{2}+(x_{2}-a_{2})^{2}}}-\frac{x_{2}-b_{2}}{\sqrt{(x_{1}-b_{1})^{2}+(x_{2}-b_{2})^{2}}}=0\\
\,\\
\Leftrightarrow & \left\{ \begin{array}{c}
(x_{2}-a_{2})^{2}(x_{1}-b_{1})^{2}=(x_{2}-b_{2})^{2}(x_{1}-a_{1})^{2}\\
(x_{2}-a_{2})(x_{2}-b_{2})\geq0
\end{array}\right.\\
\,\\
\Leftrightarrow & (x_{2}-a_{2})\sqrt{(x_{1}-b_{1})^{2}}=(x_{2}-b_{2})\sqrt{(x_{1}-a_{1})^{2}}\\
\,\\
\Leftrightarrow & (x_{2}-a_{2})\cdot|x_{1}-b_{1}|=(x_{2}-b_{2})\cdot|x_{1}-a_{1}|\\
\,\\
\Leftrightarrow & x_{2}\cdot|x_{1}-b_{1}|-a_{2}\cdot|x_{1}-b_{1}|=x_{2}\cdot|x_{1}-a_{1}|-b_{2}\cdot|x_{1}-a_{1}|\\
\,\\
\Leftrightarrow & x_{2}\cdot\left(|x_{1}-b_{1}|-|x_{1}-a_{1}|\right)=a_{2}\cdot|x_{1}-b_{1}|-b_{2}\cdot|x_{1}-a_{1}|\\
\,\\
\Leftrightarrow & x_{2}=\frac{a_{2}\cdot|x_{1}-b_{1}|-b_{2}\cdot|x_{1}-a_{1}|}{|x_{1}-b_{1}|-|x_{1}-a_{1}|}
\end{array}
\]
Since $x_{1}=x_{1}^{+}$, we conclude that $x_{2}=\varphi_{1}(x_{1}^{+},\mathbf{a},\mathbf{b})$.
It means that if an optimizer is in the interior of the right edge,
it is the point $(x_{1}^{+},\varphi_{1}(x_{1}^{+},\mathbf{a},\mathbf{b}))$.

\textbf{Case 3}. $\mathbf{x}$ is in the interior of $[\mathbf{x}]$.
We have
\[
\begin{array}{cl}
 & \left\{ \begin{array}{ccc}
\frac{\partial f}{\partial x_{1}}(\mathbf{x}) & = & 0\\
\frac{\partial f}{\partial x_{2}}(\mathbf{x}) & = & 0
\end{array}\right.\\
\,\\
\Leftrightarrow & \left\{ \begin{array}{c}
\frac{x_{1}-a_{1}}{\sqrt{(x_{1}-a_{1})^{2}+(x_{2}-a_{2})^{2}}}=\frac{x_{1}-b_{1}}{\sqrt{(x_{1}-b_{1})^{2}+(x_{2}-b_{2})^{2}}}\\
\frac{x_{2}-a_{2}}{\sqrt{(x_{1}-a_{1})^{2}+(x_{2}-a_{2})^{2}}}=\frac{x_{2}-b_{2}}{\sqrt{(x_{1}-b_{1})^{2}+(x_{2}-b_{2})^{2}}})
\end{array}\right.\\
 & \,\\
\Leftrightarrow & \left\{ \begin{array}{c}
\frac{\left(x_{1}-a_{1}\right)^{2}}{(x_{1}-a_{1})^{2}+(x_{2}-a_{2})^{2}}=\frac{\left(x_{1}-b_{1}\right)^{2}}{(x_{1}-b_{1})^{2}+(x_{2}-b_{2})^{2}}\\
\frac{\left(x_{2}-a_{2}\right)^{2}}{(x_{1}-a_{1})^{2}+(x_{2}-a_{2})^{2}}=\frac{\left(x_{2}-b_{2}\right)^{2}}{(x_{1}-b_{1})^{2}+(x_{2}-b_{2})^{2}}
\end{array}\right.\\
\,\\
\Leftrightarrow & \left\{ \begin{array}{c}
\begin{array}{l}
\left(x_{1}-a_{1}\right)^{2}(x_{2}-b_{2})^{2}=(x_{2}-a_{2})^{2}\left(x_{1}-b_{1}\right)^{2}\\
\left(x_{1}-a_{1}\right)\left(x_{1}-b_{1}\right)\geq0\\
\left(x_{2}-a_{2}\right)\left(x_{2}-b_{2}\right)\geq0
\end{array}\end{array}\right.\\
 & \,\\
\Leftrightarrow & \left\{ \begin{array}{c}
\begin{array}{l}
\left(x_{1}-a_{1}\right)(x_{2}-b_{2})=(x_{2}-a_{2})\left(x_{1}-b_{1}\right)\\
\left(x_{1}-a_{1}\right)\left(x_{1}-b_{1}\right)\geq0
\end{array}\end{array}\right.
\end{array}
\]
It means that $\mathbf{x}$ belongs to one of the two exterior half
line delimited by $\mathbf{a},\mathbf{b}$ which crosses the boundary
of $[\mathbf{x}]$. The extremum correspond to $\pm\|\mathbf{a}-\mathbf{b}\|$,
which is also reached by one element of the boundary.

\subsection{Illustration}

Consider the set

\begin{equation}
\mathbb{X}=\left\{ \mathbf{x}\in\mathbb{R}^{2}|\|\mathbf{x}-\mathbf{a}\|-\|\mathbf{x}-\mathbf{b}\|\in[y]\right\} 
\end{equation}
where $\mathbf{a}=(-1,-2)$, $\mathbf{b}=(2,3)$ and $[y]=[3,5]$.
Equivalently, we have
\begin{equation}
\mathbb{X}=f^{-1}([y])
\end{equation}
where
\begin{equation}
f(\mathbf{x})=\|\mathbf{x}-\mathbf{a}\|-\|\mathbf{x}-\mathbf{b}\|
\end{equation}
We have the following tests
\begin{equation}
\begin{array}{ccc}
f([\mathbf{x}])\cap[y]=\emptyset & \Rightarrow & [\mathbf{x}]\cap\mathbb{X}=\emptyset\\
f([\mathbf{x}])\subset[y]=\emptyset & \Rightarrow & [\mathbf{x}]\subset\mathbb{X}
\end{array}
\end{equation}
These tests can be used by a paver to approximate $\mathbb{X}$. Now,
only an outer approximation of $f([\mathbf{x}])$ can be computed.
If we use the inclusion test based on the KKT conditions, we get Figure
\ref{fig:tdoa_test_KKT.png}. We observe that only boxes that intersect
the boundary of $\mathbb{X}$ are bisected by the paver. This is due
to the fact that we have a minimal inclusion test and that $f$ is
scalar. Equivalently, we can say that we have no \emph{clustering
effect}. Using a classical interval extension \citep{Moore79}, we
get Figure \ref{fig:tdoa_test_classic} which contains more boxes
(238853 instead of 52779) for the same accuracy ($\varepsilon=0.01$).
The clustering effect is now visible. The computing time approximately
10 times smaller (less than 0.05 sec) with the KKT approach. The frame
box is $[-15,15]^{2}.$

\begin{figure}[H]
\begin{centering}
\includegraphics[width=8cm]{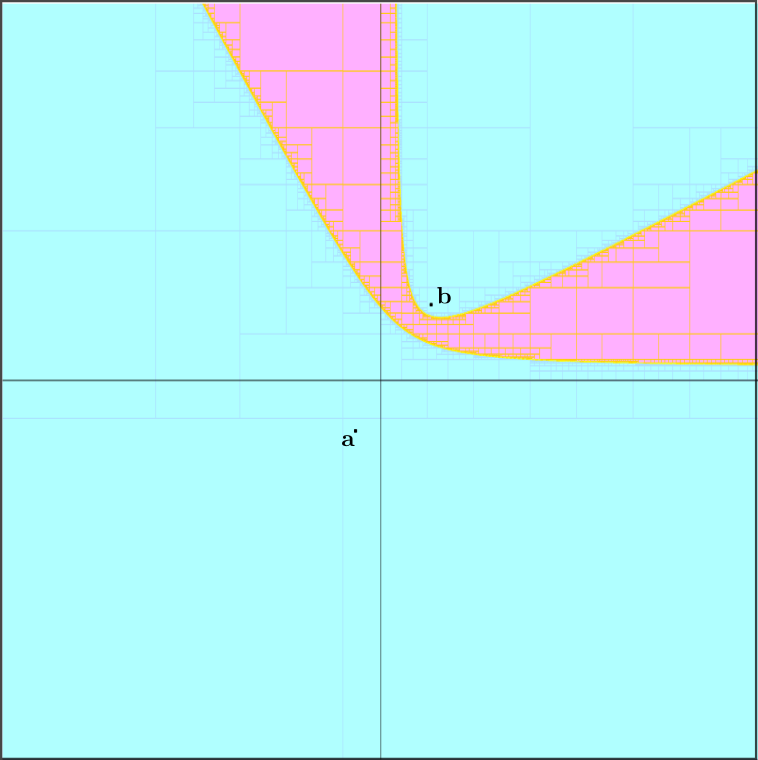}
\par\end{centering}
\caption{Set $\mathbb{X}$ obtained by the paver using the KKT conditions}
 \label{fig:tdoa_test_KKT.png}
\end{figure}

\begin{figure}[H]
\begin{centering}
\includegraphics[width=8cm]{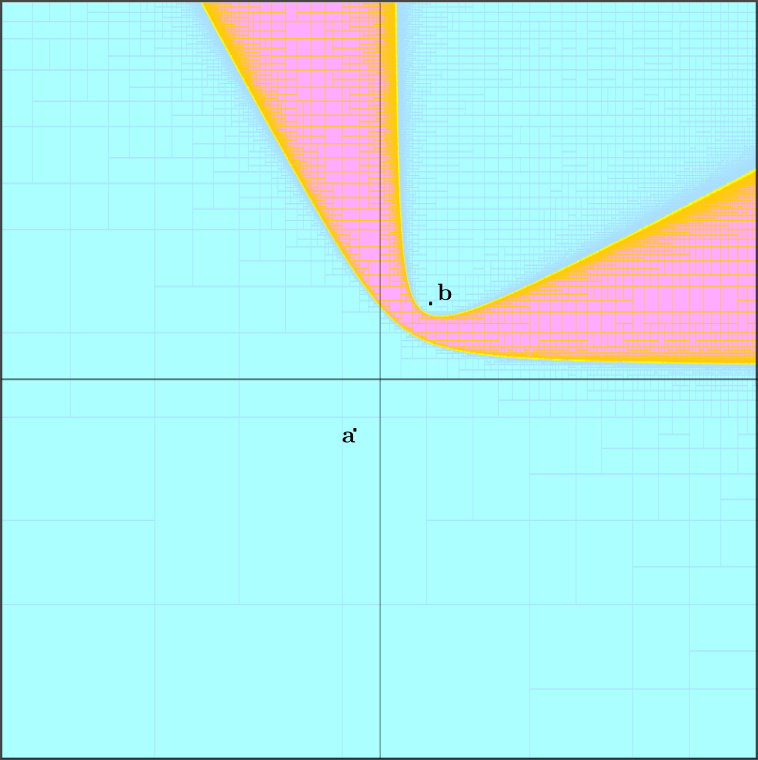}
\par\end{centering}
\caption{Set $\mathbb{X}$ obtained by the paver using a natural interval extension
for $f$}
 \label{fig:tdoa_test_classic}
\end{figure}

\subsection{Contractor from the inclusion test}

Consider a set $\mathbb{X}$ for which we have an inclusion test $[t]$
(see \citep{JaulinBook01}). Recall that an inclusion test returns
a Boolean interval, \emph{i.e.}, an element of $\mathbb{IB}=\left\{ [0,0],[0,1],[1,1]\right\} $
such that 
\begin{equation}
\begin{array}{ccc}
[t]([\mathbf{x}])=[0,0] & \Rightarrow & [\mathbf{x}]\cap\mathbb{X}=\emptyset\\{}
[t]([\mathbf{x}])=[1,1] & \Rightarrow & [\mathbf{x}]\subset\mathbb{X}=\emptyset.
\end{array}
\end{equation}
From an inclusion test $[t]$ for \noun{$\mathbb{X}$}, we can define
the contractor $\mathcal{C}_{\mathbb{X}}$ for $\mathbb{X}$ as
\begin{equation}
\begin{array}{cc}
\mathcal{C}_{\mathbb{X}}([\mathbf{x}])=\emptyset & \text{ if }[t]([\mathbf{x}])=[0,0]\\
\mathcal{C}_{\mathbb{X}}([\mathbf{x}])=[\mathbf{x}] & \text{otherwize}\text{ }
\end{array}
\end{equation}
Such a contractor is said to be \emph{binary} since it contracts a
box either to the empty set or not at all. If the test $[t]$ is minimal,
the contractor $\mathcal{C}_{\mathbb{X}}$ will not yield a clustering
effect. This shows why when we want to build an efficient contractor,
it is important to focus mostly on the forward part. 

Casting an inclusion test into a contractor allows us to use the contractor
algebra and the composition. This will be illustrated by the following
section. 

\section{Action of a contractor on a separator\label{sec:Action}}

For contractors as well for separators, classical operations of sets,
such as $\cap,\cup,\dots$ can be used. We propose here a new operation
combining contractors and separators. We will first introduce the
classical notion of correspondence (or multivalued mapping) which
can be seen as a generalization of functions. This leads us to the
notion of \emph{directed contractors} defined in \citep{jaulin2021boundary}. 

\subsection{Correspondence}

A \emph{correspondence} \citep{Aubin90}(or\emph{ binary relation})
between two sets $\mathbb{A}$ and $\mathbb{B}$ is any subset $\mathbb{F}$
of the Cartesian product $\mathbb{A}\times\mathbb{B}$. The domain
of $\mathbb{F}$ is
\begin{equation}
\text{dom}\mathbb{F}=\{a\in\mathbb{A}\,|\,\exists b,(a,b)\in\mathbb{F}\}.
\end{equation}
The range of $\mathbb{F}$ is 
\begin{equation}
\text{range}\mathbb{F}=\{b\in\mathbb{B}\,|\,\exists a,(a,b)\in\mathbb{F}\}.
\end{equation}
The image of $a\in\mathbb{A}$ by $\mathbb{F}$ is 
\begin{equation}
\mathbb{F}(a)=\{b,(a,b)\in\mathbb{F}\}
\end{equation}
and the co-image of $b$ by $\mathbb{F}$ is 
\begin{equation}
\mathbb{F}^{-1}(b)=\{a,(a,b)\in\mathbb{F}\}
\end{equation}
The \emph{inverse} of $\mathbb{F}$ is the correspondence defined
by
\begin{equation}
\mathbb{F}^{\#}=\{(b,a)|(a,b)\in\mathbb{F}\}.
\end{equation}

\subsection{Contractor for a correspondence}

Consider a contractor $\mathcal{C}_{\mathbb{F}}$ for the correspondence
$\mathbb{F}\subset\mathbb{R}^{n}\times\mathbb{R}^{p}$. We define
the \emph{forward contractor} as
\begin{equation}
\overset{\rightarrow}{\mathcal{C}}_{\mathbb{F}}^{[\mathbf{y}]}\left([\mathbf{x}]\right)=\pi_{\mathbf{y}}\circ\mathcal{C}_{\mathbb{F}}([\mathbf{x}],[\mathbf{y}])
\end{equation}
where $\pi_{\mathbf{y}}$ represents the projection onto $\mathbb{R}^{p}$
parallel to $\mathbb{R}^{n}$. The \emph{backward contractor} is defined
by
\begin{equation}
\overleftarrow{\mathcal{C}}_{\mathbb{F}}^{[\mathbf{x}]}\left([\mathbf{y}]\right)=\pi_{\mathbf{x}}\circ\mathcal{C}_{\mathbb{F}}([\mathbf{x}],[\mathbf{y}])
\end{equation}
Often, in our applications, $\mathbb{F}$ corresponds to a function
$\mathbf{f}:\mathbb{R}^{n}\mapsto\mathbb{R}^{p}$ or more precisely
to the graph of a function: $\mathbb{F}=\{(\mathbf{x},\mathbf{y})|\mathbf{y}=\mathbf{f}(\mathbf{x})\}$.

\subsection{Action}

Consider the correspondence $\mathbb{F}\subset\mathbb{R}^{n}\times\mathbb{R}^{p}$
and the set $\mathbb{Y}\subset\mathbb{R}^{p}$. We define the action
of $\mathbb{F}$ on $\mathbb{Y}$ as
\begin{equation}
\mathbb{F}\bullet\mathbb{X}=\left\{ \mathbf{y}\in\mathbb{R}^{p}|\exists\mathbf{x}\in\mathbb{X},(\mathbf{x},\mathbf{y})\in\mathbb{F}\right\} .
\end{equation}
As a consequence, 
\begin{equation}
\mathbb{F}^{\#}\bullet\mathbb{Y}=\left\{ \mathbf{x}\in\mathbb{R}^{n}|\exists\mathbf{y}\in\mathbb{Y},(\mathbf{x},\mathbf{y})\in\mathbb{F}\right\} .
\end{equation}

Note that the action used here has some similarities with the operators
used for group action \citep{olver1999}, even if the group structure
does not exist here. 
\begin{prop}
\label{prop:act:sep}Consider a separator $\mathcal{S}_{\mathbb{X}}=\left\{ \mathcal{S}_{\mathbb{X}}^{\text{in}},\mathcal{S}_{\mathbb{X}}^{\text{out}}\right\} $
for $\mathbb{X}$ and a contractor \noun{$\mathcal{C}_{\mathbb{F}}$}
for $\mathbb{F}\subset\mathbb{R}^{n}\times\mathbb{R}^{p}$. A separator
$\mathcal{S}_{\mathbb{Y}}$ for the set $\mathbb{Y}=\mathbb{F}\bullet\mathbb{X}$,
denoted by $\mathcal{S}_{\mathbb{Y}}=\mathcal{C}_{\mathbb{F}}\bullet\mathcal{S}_{\mathbb{X}}$
is:
\[
\begin{array}{ccl}
\mathcal{S}_{\mathbb{Y}}([\mathbf{y}]) & = & \mathcal{C}_{\mathbb{F}}\bullet\mathcal{S}_{\mathbb{X}}([\mathbf{y}])\\
 & = & \left\{ \mathcal{S}_{\mathbb{Y}}^{\text{in}}([\mathbf{y}]),\mathcal{S}_{\mathbb{Y}}^{\text{out}}([\mathbf{y}])\right\} \\
 & = & \left\{ \left(\left[[\mathbf{y}]\backslash\mathbb{F}\bullet\mathbb{R}^{n}\right]\right)\sqcup\overset{\rightarrow}{\mathcal{C}}_{\mathbb{F}}^{[\mathbf{y}]}\circ\mathcal{S}_{\mathbb{X}}^{\text{in}}\circ\overleftarrow{\mathcal{C}}_{\mathbb{F}}^{\mathbb{R}^{n}}([\mathbf{y}]),\overset{\rightarrow}{\mathcal{C}}_{\mathbb{F}}^{[\mathbf{y}]}\circ\mathcal{S}_{\mathbb{X}}^{\text{out}}\circ\overleftarrow{\mathcal{C}}_{\mathbb{F}}^{\mathbb{R}^{n}}([\mathbf{y}])\right\} 
\end{array}
\]
\end{prop}
In this formula, $\left[[\mathbf{y}]\backslash\mathbb{F}\bullet\mathbb{R}^{n}\right]$
represents the smallest box which encloses the set, where 
\[
[\mathbf{y}]\backslash\mathbb{F}\bullet\mathbb{R}^{n}=\left\{ \mathbf{y}\in[\mathbf{y}]\,|\,\mathbf{y}\notin\mathbb{F}\bullet\mathbb{R}^{n}\right\} .
\]
 An illustration is provided by Figures \ref{fig:FbulletIn} and \ref{fig:FbulletOut}.

\begin{figure}[H]
\begin{centering}
\includegraphics[width=11cm]{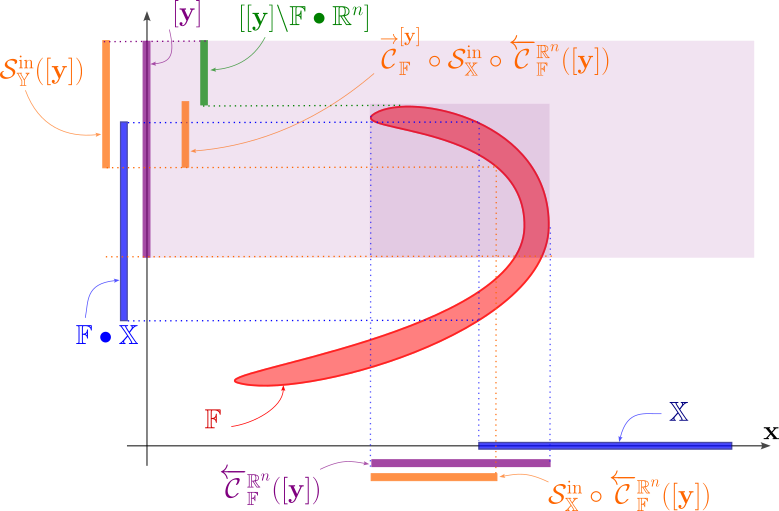}
\par\end{centering}
\caption{Inner contractor for the set $\mathbb{F}\bullet\mathbb{X}$}
\label{fig:FbulletIn}
\end{figure}

\begin{figure}[H]
\begin{centering}
\includegraphics[width=11cm]{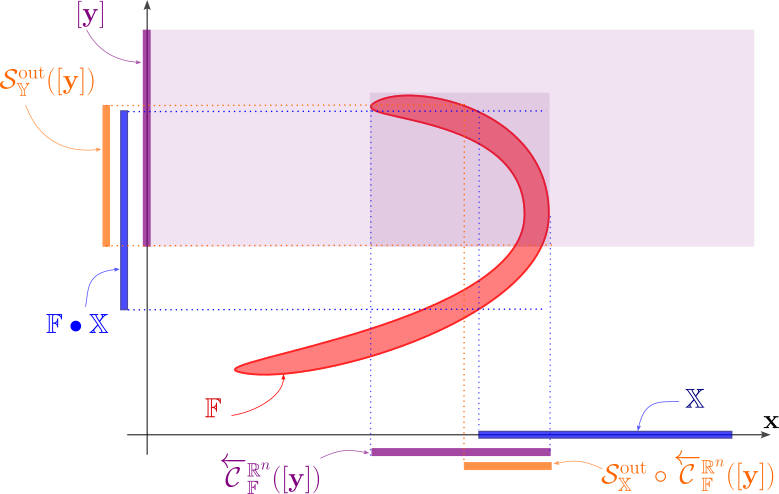}
\par\end{centering}
\caption{Outer contractor for the set $\mathbb{F}\bullet\mathbb{X}$ }
\label{fig:FbulletOut}
\end{figure}

\begin{proof}
We have
\begin{equation}
\begin{array}{ccl}
\mathcal{S}_{\mathbb{Y}}^{\text{in}}([\mathbf{y}]) & = & \left[[\mathbf{y}]\backslash\mathbb{F}\bullet\mathbb{R}^{n}\right]\sqcup\overset{\rightarrow}{\mathcal{C}}_{\mathbb{F}}^{[\mathbf{y}]}\circ\mathcal{S}_{\mathbb{X}}^{\text{in}}\circ\overleftarrow{\mathcal{C}}_{\mathbb{F}}^{\mathbb{R}^{n}}([\mathbf{y}])\\
 & = & \left[[\mathbf{y}]\backslash\mathbb{F}\bullet\mathbb{R}^{n}\right]\sqcup\overset{\rightarrow}{\mathcal{C}}_{\mathbb{F}}^{[\mathbf{y}]}\circ\mathcal{S}_{\mathbb{X}}^{\text{in}}\circ\pi_{\mathbf{x}}\left(\mathcal{C}_{\mathbb{F}}(\mathbb{R}^{n},[\mathbf{y}])\right)\\
 & \supset & \left[[\mathbf{y}]\backslash\mathbb{F}\bullet\mathbb{R}^{n}\right]\sqcup\overset{\rightarrow}{\mathcal{C}}_{\mathbb{F}}^{[\mathbf{y}]}\circ\mathcal{S}_{\mathbb{X}}^{\text{in}}\circ\pi_{\mathbf{x}}\left(\left[(\mathbf{x},\mathbf{y})|(\mathbf{x},\mathbf{y})\in\mathbb{F}\right]\right)\\
 & = & \left[[\mathbf{y}]\backslash\mathbb{F}\bullet\mathbb{R}^{n}\right]\sqcup\overset{\rightarrow}{\mathcal{C}}_{\mathbb{F}}^{[\mathbf{y}]}\circ\mathcal{S}_{\mathbb{X}}^{\text{in}}\left(\left[\mathbf{x}|\exists\mathbf{y}\in[\mathbf{y}],(\mathbf{x},\mathbf{y})\in\mathbb{F}\right]\right)\\
 & \supset & \left[[\mathbf{y}]\backslash\mathbb{F}\bullet\mathbb{R}^{n}\right]\sqcup\overset{\rightarrow}{\mathcal{C}}_{\mathbb{F}}^{[\mathbf{y}]}\left(\left[\mathbf{x}\notin\mathbb{X}|\exists\mathbf{y}\in[\mathbf{y}],(\mathbf{x},\mathbf{y})\in\mathbb{F}\right]\right)\\
 & = & \left[[\mathbf{y}]\backslash\mathbb{F}\bullet\mathbb{R}^{n}\right]\sqcup\pi_{\mathbf{y}}\circ\mathcal{C}_{\mathbb{F}}(\left[\mathbf{x}\notin\mathbb{X}|\exists\mathbf{y}\in[\mathbf{y}],(\mathbf{x},\mathbf{y})\in\mathbb{F}\right],[\mathbf{y}])\\
 & \supset & \left[[\mathbf{y}]\backslash\mathbb{F}\bullet\mathbb{R}^{n}\right]\cap\left\{ \mathbf{y}\in[\mathbf{y}]|\exists\mathbf{x}\notin\mathbb{X},(\mathbf{x},\mathbf{y})\in\mathbb{F}\right\} \\
 & = & [\mathbf{y}]\cap\left(\mathbb{R}^{p}\backslash\mathbb{F}\bullet\mathbb{X}\right)
\end{array}
\end{equation}

Moreover
\begin{equation}
\begin{array}{ccl}
\mathcal{S}_{\mathbb{Y}}^{\text{out}}([\mathbf{y}]) & = & \overset{\rightarrow}{\mathcal{C}}_{\mathbb{F}}^{[\mathbf{y}]}\circ\mathcal{S}_{\mathbb{X}}^{\text{out}}\circ\overleftarrow{\mathcal{C}}_{\mathbb{F}}^{\mathbb{R}^{n}}([\mathbf{y}])\\
 & = & \overset{\rightarrow}{\mathcal{C}}_{\mathbb{F}}^{[\mathbf{y}]}\circ\mathcal{S}_{\mathbb{X}}^{\text{out}}\circ\pi_{\mathbf{x}}\left(\mathcal{C}_{\mathbb{F}}(\mathbb{R}^{n},[\mathbf{y}])\right)\\
 & \supset & \overset{\rightarrow}{\mathcal{C}}_{\mathbb{F}}^{[\mathbf{y}]}\circ\mathcal{S}_{\mathbb{X}}^{\text{out}}\circ\pi_{\mathbf{x}}\left(\left[(\mathbf{x},\mathbf{y})|(\mathbf{x},\mathbf{y})\in\mathbb{F}\right]\right)\\
 & = & \overset{\rightarrow}{\mathcal{C}}_{\mathbb{F}}^{[\mathbf{y}]}\circ\mathcal{S}_{\mathbb{X}}^{\text{out}}\left(\left[\mathbf{x}|\exists\mathbf{y}\in[\mathbf{y}],(\mathbf{x},\mathbf{y})\in\mathbb{F}\right]\right)\\
 & \supset & \overset{\rightarrow}{\mathcal{C}}_{\mathbb{F}}^{[\mathbf{y}]}\left(\left[\mathbf{x}\in\mathbb{X}|\exists\mathbf{y}\in[\mathbf{y}],(\mathbf{x},\mathbf{y})\in\mathbb{F}\right]\right)\\
 & = & \pi_{\mathbf{y}}\circ\mathcal{C}_{\mathbb{F}}(\left[\mathbf{x}\in\mathbb{X}|\exists\mathbf{y}\in[\mathbf{y}],(\mathbf{x},\mathbf{y})\in\mathbb{F}\right],[\mathbf{y}])\\
 & \supset & \left\{ \mathbf{y}\in[\mathbf{y}]|\exists\mathbf{x}\in\mathbb{X},(\mathbf{x},\mathbf{y})\in\mathbb{F}\right\} \\
 & = & [\mathbf{y}]\cap\mathbb{F}\bullet\mathbb{X}
\end{array}
\end{equation}
\end{proof}
\begin{prop}
Consider a separator $\mathcal{S}_{\mathbb{Y}}=\left\{ \mathcal{S}_{\mathbb{Y}}^{\text{in}},\mathcal{S}_{\mathbb{Y}}^{\text{out}}\right\} $
for $\mathbb{Y}\subset\mathbb{R}^{p}$ and a contractor \noun{$\mathcal{C}_{\mathbb{F}}$}
for $\mathbb{F}\subset\mathbb{R}^{n}\times\mathbb{R}^{p}$. A separator
$\mathcal{S}_{\mathbb{X}}$ for the set $\mathbb{X}=\mathbb{F}^{\#}\bullet\mathbb{Y}$
is:

\[
\begin{array}{ccl}
\mathcal{S}_{\mathbb{X}}([\mathbf{x}]) & = & \mathcal{C}_{\mathbb{F}^{\#}}\bullet\mathcal{S}_{\mathbb{X}}([\mathbf{x}])\\
 & = & \left\{ \mathcal{S}_{\mathbb{X}}^{\text{in}}([\mathbf{x}]),\mathcal{S}_{\mathbb{X}}^{\text{out}}([\mathbf{x}])\right\} \\
 & = & \left\{ \left(\left[[\mathbf{x}]\backslash\mathbb{F}^{\#}\bullet\mathbb{R}^{p}\right]\right)\sqcup\overleftarrow{\mathcal{C}}_{\mathbb{F}}^{[\mathbf{x}]}\circ\mathcal{S}_{\mathbb{Y}}^{\text{in}}\circ\overset{\rightarrow}{\mathcal{C}}_{\mathbb{F}}^{\mathbb{R}^{p}}([\mathbf{x}]),\overleftarrow{\mathcal{C}}_{\mathbb{F}}^{[\mathbf{x}]}\circ\mathcal{S}_{\mathbb{Y}}^{\text{out}}\circ\overset{\rightarrow}{\mathcal{C}}_{\mathbb{F}}^{\mathbb{R}^{p}}([\mathbf{x}])\right\} 
\end{array}
\]
\end{prop}
\begin{proof}
It is a direct consequence of Proposition \ref{prop:act:sep}.
\end{proof}

\subsection{Illustration}

Consider the two disks
\begin{equation}
\begin{array}{ccc}
\mathbb{Y}_{1} & = & \left\{ (y_{1},y_{2})|\left(y_{1}-2\right)^{2}+\left(y_{2}-1\right)^{2}-1\leq0\right\} \\
\mathbb{Y}_{2} & = & \left\{ (y_{1},y_{2})|\left(y_{1}+1\right)^{2}+\left(y_{2}+2\right)^{2}-1\leq0\right\} 
\end{array}
\end{equation}
Consider the set $\mathbb{F}$ of all $(\mathbf{x},\mathbf{y})\subset\mathbb{R}^{n}$which
satisfy: 
\begin{equation}
\mathbb{F}:\left\{ \begin{array}{ccc}
\|\mathbf{x}-\mathbf{m}(1)\|-\|\mathbf{x}-\mathbf{m}(2)\|-y_{1} & = & 0\\
\|\mathbf{x}-\mathbf{m}(2)\|-\|\mathbf{x}-\mathbf{m}(3)\|-y_{2} & = & 0
\end{array}\right.\label{eq:fab}
\end{equation}
where 
\begin{equation}
\mathbf{m}(1)=\left(\begin{array}{c}
-1\\
-2
\end{array}\right),\,\mathbf{m}(2)=\left(\begin{array}{c}
2\\
3
\end{array}\right)\,\text{and }\mathbf{m}(3)=\left(\begin{array}{c}
4\\
1
\end{array}\right).\label{eq:m1m2m3}
\end{equation}
We want to characterize the set 
\begin{equation}
\mathbb{X}=\left\{ \mathbf{x}|\exists\mathbf{y}\in\mathbb{Y}_{1}\cup\mathbb{Y}_{2},(\mathbf{x},\mathbf{y})\in\mathbb{F}\right\} .
\end{equation}
Since 
\begin{equation}
\mathbb{X}=\mathbb{F}^{\#}\bullet(\mathbb{Y}_{1}\cup\mathbb{Y}_{2}),
\end{equation}
we get the following separator for $\mathbb{X}$:
\begin{equation}
\mathcal{S}_{\mathbb{X}}=\mathcal{C}_{\mathbb{F}^{\#}}\bullet(\mathcal{S}_{\mathbb{Y}_{1}}\cup\mathcal{S}_{\mathbb{Y}_{2}})
\end{equation}
where $\mathcal{C}_{\mathbb{F}^{\#}}$ is contractor for $\mathbb{F}$
and $\mathcal{S}_{\mathbb{Y}_{1}},\mathcal{S}_{\mathbb{Y}_{2}}$ are
separators for $\mathbb{Y}_{1},\mathbb{Y}_{2}$ . Using a paver, we
get the approximation of $\mathbb{X}$ depicted in Figure \ref{fig:tdoa_disks}.

\begin{figure}[H]
\begin{centering}
\includegraphics[width=8cm]{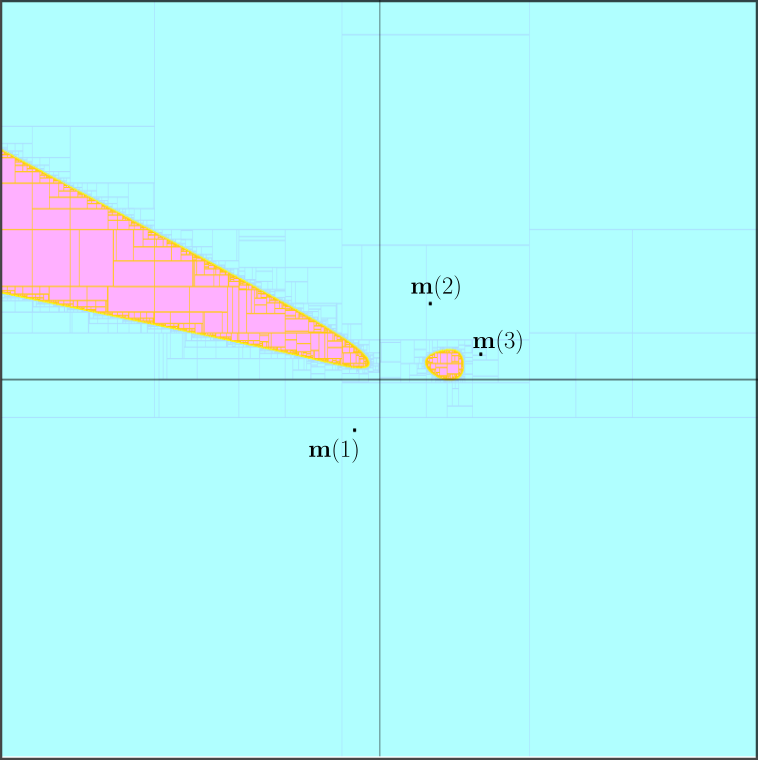}
\par\end{centering}
\caption{Approximation of the set $\mathbb{X}$ of all $\mathbf{x}$ consistent
with the two disks $\mathbb{Y}_{1},\mathbb{Y}_{2}$}
 \label{fig:tdoa_disks}
\end{figure}

\section{Application\label{sec:Application}}

Consider three microphones at positions $\mathbf{m}(1),\mathbf{m}(2),\mathbf{m}(3)$
of the plane (see (\ref{eq:m1m2m3})). They record sounds $s_{1}(t),s_{2}(t),s_{3}(t)$
of the noisy environment for a short time window. If for $\text{\ensuremath{\tau_{1},\tau_{2}}},$we
observe that $s_{1}(t),s_{2}(t+\tau_{1}),s_{3}(t+\tau_{1}+\tau_{2})$
are correlated, then we can guess the noise has possibly been emitted
from a position $\mathbf{x}$ which satisfies
\[
\left\{ \begin{array}{ccc}
\|\mathbf{x}-\mathbf{m}(1)\|-\|\mathbf{x}-\mathbf{m}(2)\|-c\tau_{1} & = & 0\\
\|\mathbf{x}-\mathbf{m}(2)\|-\|\mathbf{x}-\mathbf{m}(3)\|-c\tau_{2} & = & 0
\end{array}\right.
\]
where $c$ is the celerity of the sound. The quantity $y_{1}=c\tau_{1}$
and $y_{2}=c\tau_{2}$ are called \emph{pseudo-distances}. Using a
time-frequency analysis \citep{Cohen95}, it is possible to get a
possibility distribution \citep{Dubois80} in the pseudo-distance
plane ($y_{1},y_{2}$). 

For simplicity, assume that this possibility distribution is given
by: 
\[
\mu(\mathbf{y})=e^{-(y_{1}-2)^{2}-(y_{2}-1)^{2}}.
\]
Figure \ref{fig:tdoa_regionsY} illustrates this possibility distribution
for some $\alpha$-cuts, where
\[
\alpha_{i}=e^{-2^{i-1}},i\in\left\{ 0,\dots,5\right\} .
\]
Equivalently, the $\alpha$-cuts are defined by
\[
\begin{array}{ccc}
\mathbb{Y}_{\alpha_{i}} & = & \left\{ \mathbf{y}|e^{-(y_{1}-2)^{2}-(y_{2}-1)^{2}}\geq\alpha_{i}\right\} \\
 & = & \left\{ \mathbf{y}|(y_{1}-2)^{2}-(y_{2}-1)^{2}\leq2^{i-1}\right\} 
\end{array}
\]

The frame box is $[-10,10]^{2}.$

\begin{figure}[H]
\begin{centering}
\includegraphics[width=8cm]{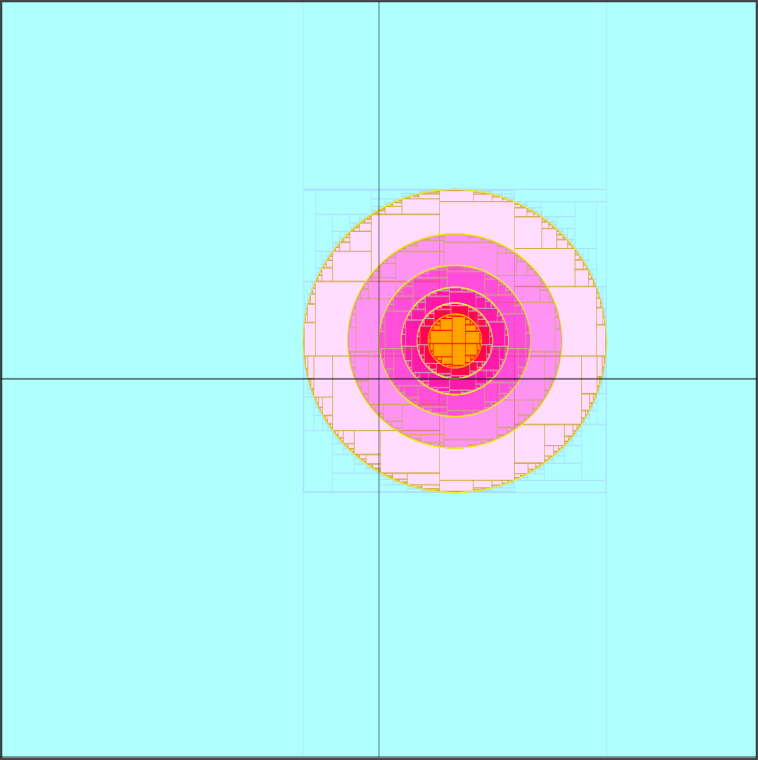}
\par\end{centering}
\caption{Possibility distribution $\mu(\mathbf{y})$ represented by its $\alpha$
cuts $\mathbb{Y}_{\alpha}$}
 \label{fig:tdoa_regionsY}
\end{figure}

For real applications, the possibility distribution has no reason
to be nested disks except maybe in the case where we have a unique
source. 

The corresponding possibility distribution for $\mathbf{x}$ is described
by the $\alpha$-cuts: 
\[
\mathbb{X}_{\alpha}=\mathbb{F}^{\#}\bullet\mathbb{Y}_{\alpha}
\]
as represented by Figure \ref{fig:tdoa_regions}. This image gives
us an idea of where the sound sources can possibly be located.

\begin{figure}[H]
\begin{centering}
\includegraphics[width=8cm]{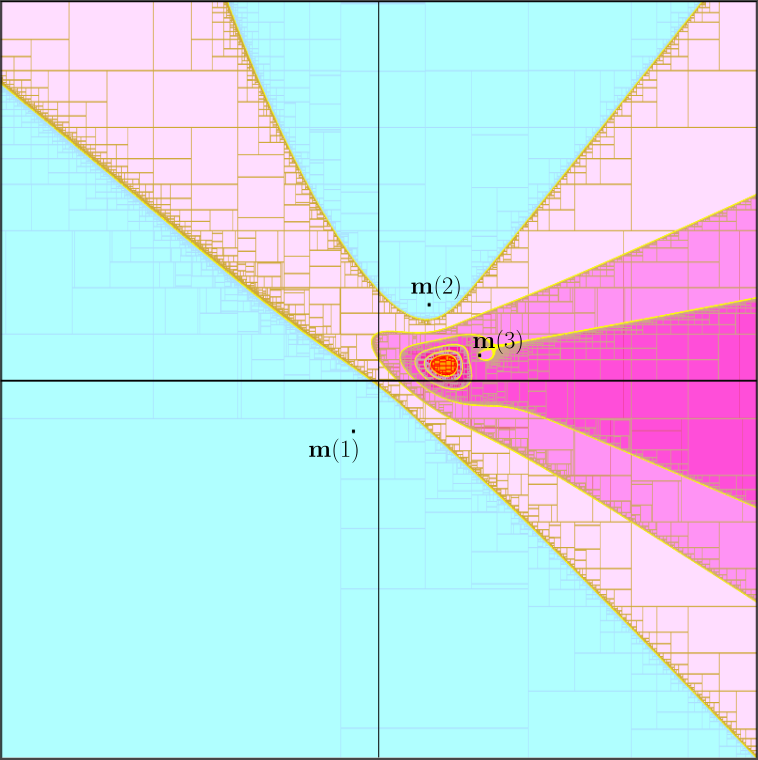}
\par\end{centering}
\caption{Possibility distribution $\mathbb{X}_{\alpha}$ for the location of
the sound sources}
 \label{fig:tdoa_regions}
\end{figure}

\section{Conclusion\label{sec:Conclusion}}

This paper has proposed to use the Karush–Kuhn–Tucker (KKT) conditions
to build efficient contractors for constraints of the form $y=f(\mathbf{x})$.
The motivating example that has been chosen considers the equation
of TDoA (Time Difference of Arrival) where the classical interval
propagation creates an unwanted pessimism due to the multi-occurences
of the variables. The KKT conditions lead us to a minimal inclusion
test. As a consequence, we were able to build a binary contractor
for the TDoA constraint with no clustering effect.

Another contribution of the paper is the definition of the action
of a contractor on a separator. This operation allowed us to build
separators by composition. The separator algebra, as defined in \citep{DesrochersEAAI2014},
extended set operations such as the intersection, the union or the
complement, to separators, but without the possibility the compose
the separators. We have shown that the compositions should not be
performed between separators, but between contractors and separators. 

The application that has been considered illustrates that a possibility
distribution can easily and efficiently be inverted through the TDoA
contractor to localize sources in a noisy environment. 

The Python code based on Codac \citep{codac} is given in \citep{jaulin:code:ctctdoa:23}.

\bibliographystyle{plain}

\end{document}